\documentclass[preprint,12pt,times]{article}
\usepackage{graphicx}
\usepackage{amsthm}
\usepackage{amssymb}
\usepackage{amsmath}
\usepackage{amscd}
\usepackage{bbm}
\usepackage{float}
\usepackage{booktabs}
\numberwithin{equation}{section}

\newtheorem{theorem}{Theorem}[section]
\newtheorem{lemma}{Lemma}[section]
\newtheorem{remark}{Remark}[section]

\usepackage[numbers,sort&compress]{natbib}

\DeclareMathOperator{\sgn}{sgn}

\numberwithin{figure}{section}
\numberwithin{table}{section}

\allowdisplaybreaks[2]
\newcommand{\ep}{\mathrm{e}}

\newcommand\btd{\raise 2pt \hbox{$\hat\bigtriangledown$}\hskip 1.5pt}
\newcommand\bt{\raise 2pt \hbox{$\bigtriangledown$}\hskip 1.5pt}

\begin{document}
\date{}
\title{Wave breaking for the Kakutani-Matsuuchi model  }
\author{ Shaojie Yang\thanks{Corresponding author:
shaojieyang@kust.edu.cn},~~~~Jianmin Zhao \\~\\
\small ~ Department of Systems Science and Applied Mathematics, \\
 \small~Kunming University of Science and Technology,  \\
\small ~Kunming, Yunnan 650500, China}

\date{}
\maketitle
\begin{abstract}
In this paper, we consider the Kakutani-Matsuuchi model which describes the surface elevation of the water-waves under the effect of viscosity. We show wave breaking for the Kakutani-Matsuuchi model, namely, the solution remains bounded but its slope becomes unbounded in finite time, the slope of the initial data is sufficiently negative.\\


\end{abstract}
\noindent\rule{15.5cm}{0.5pt}

\section{Introduction}\label{sec1}
The motion of gravity water waves is a hot research topic that has attracted a lot of attention from a number of different researchers in mathematics, physics and engineering. When considering gravity water waves in deep water, a classical method is to take irrotational incompressible inviscid fluids \cite{r19}. However, the effect of viscosity on gravity water waves has also attracted attention to deal with more realistic models. Examples of the necessity of effect of viscosity on gravity water waves in experiments can refer to Ref.\cite{r1}. Various approaches have been developed to modeling the effect of viscosity on gravity waves. For example, in Ref.\cite{r1,r2}, authors have derived, independently, asymptotical models for long gravity waves on viscous water waves. In a recent paper, Dias, Dyachenko and Zakharov \cite{r4} derived a nonlinear model for the motion of a surface wave under gravity and viscous effects. For more properties of Dias-Dyachenko-Zakharov's model, the reader is referred to \cite{r5,r6,r8,r9,r10} and the references therein.

In the following, we recall the model equation which appeared in the seminal paper of Kakutani and Matsuuchi \cite{r7}. There are three models have been derived in Ref.\cite{r7} according to a competition between geometrical dispersion  and dispersion provided by the viscous boundary layer (the nonlocal term). Denoting $k$ as the wave number of the long wave, the different regimes read as follows.\\
\noindent$\bullet $ $\mu\ll k^5$, viscosity effects can be neglected, the model is the KdV equation
\begin{equation}
u_t+\frac{3}{2}uu_x+\frac{1}{6}u_{xxx}=0.
\end{equation}
\noindent$\bullet $ $\mu\sim k^5$, a balance between the geometrical and the viscous dispersion, the model reads
\begin{equation}
u_t+\frac{3}{2}uu_x+\frac{1}{6}u_{xxx}=\frac{1}{4\sqrt{\pi R^*}}\int_{\mathbb{R}} \frac{\left(1-\sgn(x-y)\right)\eta_y}{\sqrt{|x-y|}}.
\end{equation}
\noindent$\bullet $ $\mu \gg k^5$, a large viscous effect, the model reads
\begin{equation}\label{A1}
u_t+\frac{3}{2}uu_x=\frac{1}{4\sqrt{\pi R^*}}\int_{\mathbb{R}} \frac{\left(1-\sgn(x-y)\right)\eta_y}{\sqrt{|x-y|}}.
\end{equation}
In this paper, we consider Eq.\eqref{A1} which has a large viscous effect.  Eq.\eqref{A1} can be rewritten as 
\begin{equation}\label{A}
u_t+uu_x+\Lambda^{\frac{1}{2}}u+\mathcal{H}\Lambda^{\frac{1}{2}}u=0,
\end{equation}
where operators $\Lambda^{\frac{1}{2}}$  and $\mathcal{H}\Lambda^{\frac{1}{2}}$ can be written as 
\begin{equation*}
\Lambda^{\frac{1}{2}}f(x)=\text{p.v.}\int_{\mathbb{R}}\frac{f(x)-f(y)}{|x-y|^{\frac{3}{2}}}dy,
\end{equation*}
\begin{equation*}
\mathcal{H}\Lambda^{\frac{1}{2}}f(x)=\text{p.v.}\int_{\mathbb{R}}\frac{f(x)-f(y)}{|x-y|^{\frac{3}{2}}}\sgn(x-y)dy
\end{equation*}
up to multiplication by constants. Recently, Chen, Dumont and Goubet \cite{r17} derived decay of solutions to \eqref{A}.  Bae,  Lee and Shin \cite{r16} showed the formation of singularities of smooth solutions in finite time for a certain class of initial data. 

Note that the equation \eqref{A} can be regard as combination of the Whitham equation with fractional dispersion
\begin{equation}\label{B}
u_t+uu_x+\Lambda^{\frac{1}{2}}u=0
\end{equation}
and the Whitham equation with fractional diffusion
\begin{equation}\label{C}
u_t+uu_x+\mathcal{H}\Lambda^{\frac{1}{2}}u=0.
\end{equation}
Eq.\eqref{B} is an extension of the Burgers equation (also called the  fractal Burgers  equation) \cite{r21}. Finite time singularities and global well-posedness of Eq.\eqref{B} have studied in \cite{r22,r23}.  Eq.\eqref{C} arises as a quadratic approximation of the water wave problem on the moving surface of a two-dimensional, infinitely deep flow under gravity \cite{r11}.
VM Hur observed that \eqref{C} shares the dispersion relation and scaling symmetry in common with water waves in the infinite depth \cite{r12}, and derived the solution of  \eqref{C} blowup in finite time \cite{r12,r13,r20}.

The purpose of this paper is to show wave breaking for Eq.\eqref{A} by using the arguments in Ref.\cite{r13,r14,r15,r18}.  The idea in Ref.\cite{r13,r14,r15,r18} is to analysis ordinary differential equations for the solution and its derivatives of all orders along the characteristics, which by the way involve nonlocal forcing terms. The main difficulty comes from loss of derivatives when handling nonlocal forcing term in estimating the sup norm of gradient of the solution.  Below we state  main result of wave breaking for Eq.\eqref{A}. 
\begin{theorem}\label{b1}
For $\epsilon>0$ sufficiently small. If $u_0\in H^\infty (\mathbb{R})$ satisfies that
\begin{align}\label{a1}
\epsilon^2(-\inf_{x\in\mathbb{R}}u^\prime_0(x))^2>1+2\|u_0\|_{H^3(\mathbb{R})},
\end{align}
\begin{align}\label{c5}
\epsilon^2(1-\epsilon)^4(-\inf_{x\in\mathbb{R}}u^\prime_0(x))^{3/4}>28(1+(1+\ep^2+\ep^{1/g})g+g^2),
\end{align}
\begin{align}\label{d8}
\epsilon^2(-\inf_{x\in\mathbb{R}}u^\prime_0(x))^{1/4}>\frac{9}{4}\ep,
\end{align}
and
\begin{align}\label{b3}
\|u_0^{(n)}\|_{L^\infty(\mathbb{R})}\leq((n-1)g)^{2(n-1)}~~for~~n=2,3,...
\end{align}
for some $g\geq 1$, then the solution of the initial value problem associated with \eqref{A} and $u(x, 0)=u_0$ exhibits
wave breaking, i.e.
\begin{align}
\notag |u(x)|<\infty~~~for~~any~~x\in \mathbb{R}~~ for~any~~t\in[0,T),
\end{align}
but
\begin{align}
\notag \inf_{x\in\mathbb{R}}u_x(t,x)\rightarrow -\infty ~~~as~~t\rightarrow T^{-},
\end{align}
for some $T>0$. Moreover
\begin{align}\label{g6}
-\frac{1}{1+\epsilon}\frac{1}{\inf\limits_{x\in \mathbb{R}}u_0^\prime(x)}<T<-\frac{1}{(1-\epsilon)^2}\frac{1}{\inf\limits_{x\in \mathbb{R}}u_0^\prime(x)}.
\end{align}
\end{theorem}

\begin{remark}
The hypotheses  \eqref{a1}-\eqref{d8} require that $u_0'$ be sufficiently negative somewhere in $\mathbb{R}$. The idea of the proofs lies in that the profile of u steepens until it becomes vertical in finite time.
\end{remark}

\section{Proof of Theorem \ref{b1}}

Assume that the initial value problem associated with \eqref{A} and $u(\cdot,x)=u_0$ possesses a unique solution in $C^\infty([0,T);H^\infty(\mathbb{R}))$
for some $T>0$. As a matter of fact, one may combine an a priori bound and a compactness argument to work out the local-in-time well-posedness in $H^s(\mathbb{R})$
for $s>3/2$ (see Ref.\cite{r24} for details). Assume that $T$ is the maximal existence time.

For $x\in \mathbb{R}$,  let $X(t,x)$ solve 
\begin{equation}\label{e3}
\begin{cases}
\frac{dX}{dt}(t,x)=u(X(t;x),t),\\
X(0;t)=x.
\end{cases}
\end{equation}
Since $u(t,x)$ is bounded and satisfies a Lipschitz condition in $x$ for any  $(t,x)\in [0,T)\times\mathbb{R}$,  it follows from the classical ODE theory that $X(\cdot;t)$ exists throughout the interval $[0,T)$ for any $x\in \mathbb{R}$. Furthermore,  $x\mapsto X(\cdot;t)$ is continuously differentiable throughout the interval $(0,T)$ for any $x\in \mathbb{R}$.

Let
\begin{align}\label{e4}
v_n(t;x)=(\partial_x^nu)(X(t;x),t)~~{\rm for}~~n=0, 1, 2, ...
\end{align}
and 
\begin{align}\label{a6}
m(t)=\inf_{x\in \mathbb{R}}v_1(t;x)=\inf_{x\in \mathbb{R}}(\partial_xu)(X(t;x),t)=:m(0)q^{-1}(t).
\end{align}
Apparently
\begin{align}
m(t)<0~~{\rm for~any}~~t\in [0,T),~~q(0)=1~~{\rm and}~~q(t)>0~~{\rm for~any}~~t\in [0,T).
\end{align}
Indeed $m(t)\geq 0$ would imply that $u(\cdot,t)$ be non-decreasing in $\mathbb{R}$, and hence $u(\cdot, t)\equiv0$.
For $x\in \mathbb{R}$, differentiating \eqref{A} with respect to $x$ and evaluating at $x=X(t;x)$, we have
\begin{equation}\label{d4}
\frac{dv_n}{dt}+\sum^n_{j=1}\binom{n}{j}v_jv_{n+1-j}+K_n(t; x)+\phi_n(t; x)=0~~for~~n=2,3,...,
\end{equation}
\begin{equation}\label{a3}
\frac{dv_1}{dt}+v_1^2+K_1(t; x)+\phi_1(t; x)=0
\end{equation}
and
\begin{align}\label{c2}
\frac{dv_0}{dt}+K_0(t; x)+\phi_0(t; x)=0,
\end{align}
where$\binom{n}{j}$ means a binomial coefficient and
\begin{align}
\notag K_n(t;x)=&(\mathcal{H}\Lambda^{\frac{1}{2}}\partial_x^nu)(X(t;x),t)\\
\notag =&\int_{-\infty}^{+\infty}\frac{\sgn(X(t;x)-y)}{|X(t;x)-y|^{\frac{3}{2}}}((\partial_x^nu)(X(t;x),t)-(\partial_x^nu)(y,t))dy,
\end{align}
\begin{align}
\notag \phi_n(t;x)=&(\Lambda^{\frac{1}{2}}\partial_x^nu)(X(t;x),t)\\
\notag =&\int_{-\infty}^{+\infty}\frac{(\partial_x^nu)(X(t;x),t)-(\partial_x^nu)(y,t)}{|X(t;x)-y|^{\frac{3}{2}}}dy
\end{align}
for $n = 0, 1, 2,...$. Let $\delta> 0$.   Splitting the integral and perform an integration by parts, one gets
\begin{align}
\notag |K_n(t;x)|=&\left|\left(\int_{|y|<\delta}+\int_{|y|>\delta}\right)\frac{\sgn(y)}{|y|^{\frac{3}{2}}}((\partial_x^nu)(X(t;x),t)-(\partial_x^nu)(X(t;x)-y,t))dy\right|\\
\notag \leq & \left|2\delta^{-\frac{1}{2}}((\partial_x^nu)(X(t;x)-\delta,t)-(\partial_x^nu)(X(t;x)+\delta,t))\right|\\
\notag &+2\left|\int_{|y|<\delta}\frac{(\partial_x^{n+1}u)(X(t;x)-y,t)}{|y|^{\frac{1}{2}}}\right|\\
\notag &+\left|\int_{|y|>\delta}\frac{\sgn(y)}{|y|^{\frac{3}{2}}}((\partial_x^nu)(X(t;x),t)-(\partial_x^nu)(X(t;x)-y,t))dy\right|\\
\notag \leq&12\delta^{-\frac{1}{2}}\|v_n(t)\|_{L^\infty}+8\delta^{\frac{1}{2}}\|v_{n+1}(t)\|_{L^\infty}
\end{align}
and
\begin{align}
\notag |\phi_n(t;x)|=&\left|\left(\int_{|y|<\delta}+\int_{|y|>\delta}\right)\frac{(\partial_x^nu)(X(t;x),t)-(\partial_x^nu)(X(t;x)-y,t)}{|y|^{\frac{3}{2}}}dy\right|\\
\notag\leq &\left|\int_{|y|<\delta}\frac{\sgn^2(y)}{|y|^{\frac{3}{2}}}((\partial_x^nu)(X(t;x),t)-(\partial_x^nu)(X(t;x)-y,t))dy \right|\\
\notag &+\left|\int_{|y|>\delta}\frac{(\partial_x^nu)(X(t;x),t)-(\partial_x^nu)(X(t;x)-y,t)}{|y|^{\frac{3}{2}}}dy\right|\\
\notag\leq &\left|\int_{|y|<\delta}\sgn(y)((\partial_x^nu)(X(t;x),t)-(\partial_x^nu)(X(t;x)-y,t))d(-\frac{2}{|y|^{\frac{1}{2}}})\right|\\
\notag &+8\delta^{-\frac{1}{2}}\|v_n(t)\|_{L^\infty}\\
\notag\leq &8\delta^{-\frac{1}{2}}\|v_n(t)\|_{L^\infty}+2\left|\frac{\sgn(y)}{|y|^{\frac{1}{2}}}((\partial_x^nu)(X(t;x),t)-(\partial_x^nu)(X(t;x)-y,t))\bigg|
_{-\delta}^{\delta}\right|\\
\notag &+2\left|\int_{-\delta}^{0}-\frac{(\partial_x^{n+1}u)(X(t;x)-y,t)}{|y|^{\frac{1}{2}}}dy\right|+2\left|\int_{0}^{\delta}\frac{(\partial_x^{n+1}u)(X(t;x)-y,t)}
{|y|^{\frac{1}{2}}}dy\right|\\
\notag\leq &16\delta^{-\frac{1}{2}}\|v_n(t)\|_{L^\infty}+8\delta^{\frac{1}{2}}\|v_{n+1}(t)\|_{L^\infty},
\end{align}
then we have
\begin{align}\label{b7}
|K_n(t;x)+\phi_n(t;x)|\leq28(\delta^{-\frac{1}{2}}\|v_n(t)\|_{L^\infty}+\delta^{\frac{1}{2}}\|v_{n+1}(t)\|_{L^\infty})
\end{align}
for $n = 0, 1, 2,...$ and for any $(t,x)\in[0,T)\times \mathbb{R}$.

Next, we shall show that
\begin{align}\label{g5}
|K_1(t;x)+\phi_1(t;x)|\leq\epsilon^2m^2(t)~~{\rm for~any}~~(t,x)\in[0,T)\times \mathbb{R}.
\end{align}
Note that \eqref{a1},  and using the Sobolev's inequality, one has
\begin{align}
|K_1(0;x)+\phi_1(0;x)|\leq|\mathcal{H}\Lambda^{\frac{1}{2}}u_0^\prime+\Lambda^{\frac{1}{2}}u_0^\prime|
\leq2\|u_0\|_{H^{2+}}<\epsilon^2m^2(0)~~{\rm for~any}~~x\in\mathbb{R}.
\end{align}
Suppose on the contrary that $|K_1(T_1;x)+\phi_1(T_1;x)|=\epsilon^2m^2(T_1)$ for some $T_1\in[0,T)$ for some $x\in\mathbb{R}$. By continuity, without loss of generality, we may assume that
\begin{align}\label{a2}
|K_1(t;x)+\phi_1(t;x)|\leq\epsilon^2m^2(t)~~{\rm for~any}~~t\in[0,T_1]~~{\rm for~any}~~(t,x)\in[0,T)\times \mathbb{R}.
\end{align}
We seek a contradiction.

\begin{lemma}\label{a7}
For $0<\gamma<1$ and for $t\in[0,T_1]$, let
\begin{align}\label{d9}
\Sigma_\gamma(t)=\{x\in\mathbb{R}: v_1(t;x)\leq(1-\gamma)m(t)\}.
\end{align}
If $0<\epsilon\leq\gamma<1/2$ for $\epsilon>0$ sufficiently small then $\Sigma_\gamma(t_2)\subset\Sigma_\gamma(t_1)$ whenever $0\leq t_1\leq t_2\leq T_1$.
\end{lemma}
\begin{proof}
The proof of Lemma \ref{a7} can be found in Ref.\cite{r14}, here we include the detail for completeness.

Suppose on the contrary that $x_1\notin\sum_\gamma(t_1)$ but $x_1\in\sum_\gamma(t_2)$  for some $x_1\in\mathbb{R}$ for some $0\leq t_1\leq t_2\leq T_1$, that is 
\begin{align}\label{a5}
v_1(t_1;x_1)>(1-\gamma)m(t_1) ~~{\rm and}~~v_1(t_2;x_1)\leq(1-\gamma)m(t_2)<\frac{1}{2}m(t_2).
\end{align}
We may choose $t_1$ and $t_2$ close so that
\begin{align}
\notag v_1(t;x_1)\leq\frac{1}{2}m(t) ~~{\rm for~any}~~t\in[t_1,t_2].
\end{align}
Indeed, $v_1(\cdot;x_1)$ and $m$ are uniformly continuous throughout the interval $[0,T_1]$.  Let
\begin{align}\label{a4}
v_1(t_1;x_2)=m(t_1)<\frac{1}{2}m(t_1).
\end{align}
We may necessarily choose $t_2$ closer to $t_1$ so that
\begin{align}
\notag v_1(t;x_2)<\frac{1}{2}m(t)~~{\rm for~any}~~t\in[t_1,t_2].
\end{align}
For $\epsilon>0$  sufficiently small, it follows from \eqref{a2} that
\begin{align}
\notag |K_1(t;x_j)+\phi_1(t;x_j)|&\leq\epsilon^2m^2(t)\leq4\epsilon^2v_1^2(t;x_j)\\
\notag&<\frac{\gamma}{2}v_1^2(t;x_j)~~{\rm for~any}~~t\in[t_1,t_2]~~{\rm and}~~j=1,2.
\end{align}
To proceed, from \eqref{a3}, we have 
\begin{align}
\notag \frac{dv_1}{dt}(\cdot; x_1)=-v_1^2(\cdot; x_1)-K_1(\cdot; x_1)-\phi_1(\cdot;x_1)\geq(-1-\frac{\gamma}{2})v_1^2(\cdot; x_1)
\end{align}
and
\begin{align}
\notag \frac{dv_1}{dt}(\cdot; x_2)\geq\left(-1+\frac{\gamma}{2}\right)v_1^2(\cdot; x_2)
\end{align}
throughout the interval $(t_1,t_2)$. Integrating them over the interval $[t_1,t_2]$, we have
\begin{align}
\notag v_1(t_2;x_1)\geq\frac{v_1(t_1;x_1)}{1+(1+\frac{\gamma}{2})v_1(t_1;x_1)(t_2-t_1)}
\end{align}
and
\begin{align}
\notag v_1(t_2;x_2)\leq\frac{v_1(t_1;x_2)}{1+(1-\frac{\gamma}{2})v_1(t_1;x_2)(t_2-t_1)}.
\end{align}
The latter inequality and \eqref{a4} imply that
\begin{align}
\notag m(t_2)\leq\frac{m(t_1)}{1+(1-\frac{\gamma}{2})m(t_1)(t_2-t_1)}.
\end{align}
The former inequality and \eqref{a5}imply that
\begin{align}
\notag v_1(t_2;x_1)>&\frac{(1-\gamma)m(t_1)}{1+(1+\frac{\gamma}{2})(1-\gamma)m(t_1)(t_2-t_1)}\\
\notag >&\frac{(1-\gamma)m(t_1)}{1+(1-\frac{\gamma}{2})m(t_1)(t_2-t_1)}\\
\notag \geq &(1-\gamma)m(t_2).
\end{align}
There is a contradiction,  therefore we completes the proof.
\end{proof}

\begin{lemma}\label{c4}
 $0<q(t)\leq1$ and it is decreasing for any $t\in[0,T_1]$.
\end{lemma}
\begin{proof}
The proof is very similar to that of \cite{r14},  here we only include the details for future usefulness.\\

Let $x\in\sum_\gamma(T_1)$, where $0<\epsilon\leq\gamma<1/2$ for $\epsilon>0$ sufficiently small. We suppress it for simplicity of notation. Note from
\eqref{a6} and Lemma \ref{a7} that
\begin{align}\label{a8}
m(t)\leq v_1(t)\leq (1-\gamma)m(t)~~{\rm for~any}~~t\in[0,T_1].
\end{align}
Let's write the solution of \eqref{a3} as
\begin{align}\label{a9}
v_1(t)=\frac{v_1(0)}{1+v_1(0)\int_{0}^{t}(1+(v_1^{-2}(K_1+\phi_1))(\tau))d\tau}=:m(0)r^{-1}(t).
\end{align}
Clearly, $r(t)>0$ for any $t\in[0,T_1]$. Note from \eqref{a8} and \eqref{a2} that
\begin{align}
\notag |(v_1^{-2}(K_1+\phi_1))(t)|<(1-\gamma)^{-2}\epsilon^2<\epsilon ~~{\rm for~any}~~t\in[0,T_1]
\end{align}
for $\epsilon>0$  sufficiently small. Therefore, it follows from \eqref{a9} that
\begin{align}\label{e8}
(1+\epsilon)m(0)\leq\frac{dr}{dt}\leq(1-\epsilon)m(0)~~{\rm throughout~the~interval}~~(0,T_1).
\end{align}
Consequently, $r(t)$ and, hence, $v_1(t)$ (see \eqref{a9}) are decreasing for any $t\in[0,T_1]$. Furthermore, $m(t)$ and, hence, $q(t)$ (see \eqref{a6})
are decreasing for any $t\in[0,T_1]$. This completes the proof. It follows from \eqref{a6}, \eqref{a9} and \eqref{a8} that
\begin{align}\label{f1}
q(t)\leq r(t)\leq \frac{1}{1-\gamma}q(t)~~{\rm for~any}~~t\in[0,T_1].
\end{align}
\end{proof}
\begin{lemma}
For $s>0,s\neq1$, and for $t\in[0,T_1]$,
\begin{align}\label{c3}
\int_{0}^{t}q^{-s}(\tau)d\tau\leq-\frac{1}{s-1}\frac{1}{(1-\epsilon)^{1+s}}\frac{1}{m(0)}\left(q^{1-s}(t)-\frac{1}{(1-\epsilon)^{1-s}}\right).
\end{align}
\end{lemma}
The proof can be  found in Ref.\cite{r2}, for instance. Hence, we omit it details. See instead the proof of \eqref{d1} below.

\begin{lemma}
For $n\geq3$,
\begin{align}\label{d6}
\sum_{j=2}^{n-1}\binom{n}{j}(j-1)^{2(j-1)}(n-j)^{2(n-j)}\leq\frac{3}{2}\ep n(n-1)^{2(n-1)}.
\end{align}
\end{lemma}
\begin{proof}
We use Stirling's inequality to compute that
\begin{align}
\notag &\sum_{j=2}^{n-1}\binom{n}{j}(j-1)^{2(j-1)}(n-j)^{2(n-j)}\\
\notag \leq &\sum_{j=2}^{n-1}\frac{n^n}{j^j(n-j)^{n-j}}(j-1)^{2(j-1)}(n-j)^{2(n-j)}\\
\notag =&n\left(\frac{n}{n-1}\right)^{n-1}(n-1)^{2(n-1)}\sum_{j=2}^{n-1}\frac{1}{j}\left(\frac{j-1}{j}\right)^{j-1}\frac{(j-1)^{j-1}(n-j)^{n-j}}{(n-1)^{n-1}}\\
\notag \leq &\ep n(n-1)^{2(n-1)}\sum_{j=2}^{n-1}\frac{1}{j}\frac{j-1}{n-1}\\
\notag \leq &\ep n(n-1)^{2(n-1)}\frac{1}{n-1}\int_{1}^{n}\frac{1}{y}dy\\
\notag \leq &\frac{3}{2}\ep n(n-1)^{2(n-1)}.
\end{align}
\end{proof}
We claim that
\begin{align}\label{b4}
\|v_0(t;\cdot)\|_{L^\infty(\mathbb{R})}=\|u(\cdot,t)\|_{L^\infty(\mathbb{R})}<C_0,
\end{align}
\begin{align}\label{b5}
\|v_1(t;\cdot)\|_{L^\infty(\mathbb{R})}=\|(\partial_xu)(\cdot,t)\|_{L^\infty(\mathbb{R})}<C_1q^{-1}(t),
\end{align}
\begin{align}\label{b6}
\|v_n(t;\cdot)\|_{L^\infty(\mathbb{R})}=\|(\partial_x^nu)(\cdot,t)\|_{L^\infty(\mathbb{R})}<C_2((n-1)g)^{2(n-1)}q^{-1-(n-1)\sigma}(t),
\end{align}
for $n=2,3,...$, for any $t\in[0,T_1]$, where
\begin{align}\label{b2}
C_0=2(\|u_0\|_{L^\infty(\mathbb{R})}+\|u_0^\prime\|)_{L^\infty(\mathbb{R})}),~~C_1=2\|u_0^\prime\|)_{L^\infty(\mathbb{R})},~~
C_2=(-m(0))^{3/4}
\end{align}
and
\begin{align}\label{c8}
\sigma=\frac{3}{2}+6\epsilon~~{\rm so~that}~~\sigma<2-20\epsilon
\end{align}
for $\epsilon$ in Theorem \ref{b1}. Note from \eqref{a1} that
\begin{align}
\notag \frac{1}{2}C_1=\|u_0^\prime\|_{L^\infty(\mathbb{R})}>C_2>1,
\end{align}
and we tacitly exercise it throughout the proof. It follows from \eqref{b2}, \eqref{a6} and \eqref{b3}, \eqref{a1} that
\begin{align}
\notag \|v_0(0;\cdot)\|_{L^\infty(\mathbb{R})}=\|u_0\|_{L^\infty(\mathbb{R})}<C_0,
\end{align}
\begin{align}
\notag \|v_1(0;\cdot)\|_{L^\infty(\mathbb{R})}=\|u_0^\prime\|_{L^\infty(\mathbb{R})}<C_1q^{-1}(0),
\end{align}
\begin{align}
\notag \|v_n(0;\cdot)\|_{L^\infty(\mathbb{R})}=\|u_0^{(n)}\|_{L^\infty(\mathbb{R})}<C_2((n-1)g)^{2(n-1)}q^{-1-(n-1)\sigma}(0)
\end{align}
for $n=2,3,...$. In other words, \eqref{b4}-\eqref{b6} hold for any $n=0,1,2,...$ at $t=0$. Suppose on the contrary that \eqref{b4}-\eqref{b6} hold for any $n=0,1,2,...$ throughout the interval $[0,T_2)$  but do not for some $n\geq0$ at $t=T_2$ for some $T_2\in(0,T_1]$.
By continuity, we find that
\begin{align}\label{b8}
\|v_0(t;\cdot)\|_{L^\infty(\mathbb{R})}\leq C_0,
\end{align}
\begin{align}\label{b9}
\|v_1(t;\cdot)\|_{L^\infty(\mathbb{R})}\leq C_1q^{-1}(t),
\end{align}
\begin{align}\label{c6}
\|v_n(t;\cdot)\|_{L^\infty(\mathbb{R})}\leq C_2((n-1)g)^{2(n-1)}q^{-1-(n-1)\sigma}(t)
\end{align}
for $n=2,3,...$ for any $t\in[0,T_2]$. We seek a contradiction.

For $n=0$, the proof is similar to that in \cite{r14}, here we include the details for future usefulness.

It follows from \eqref{b7}, where $\delta(t)=q(t)$, and \eqref{b8}-\eqref{b9} that
\begin{align}\label{c1}
|K_0(t;x)+\phi_0(t;x)|\leq 28(C_0q^{-\frac{1}{2}}(t)+C_1q^{-1}(t)q^{\frac{1}{2}}(t))=28(C_0+C_1)q^{-\frac{1}{2}}(t)
\end{align}
for any $t\in[0,T_2]$ for any $x\in\mathbb{R}$. Integrating \eqref{c2} over the interval $[0,T_2]$, we then show that
\begin{align}
|v_0(T_2;x)|\leq&\|u_0\|_{L^\infty(\mathbb{R})}+\int_{0}^{T_2}|K_0(t;x)+\phi_0(t;x)|dt\\
\notag \leq&\frac{1}{2}C_0+28(C_0+C_1)\int_{0}^{T_2}q^{-\frac{1}{2}}(t)dt\\
\notag \leq&\frac{1}{2}C_0-28(C_0+C_1)\frac{2}{(1-\epsilon)^{\frac{3}{2}}}\frac{1}{m(0)}\left(\frac{1}{(1-\epsilon)^{\frac{1}{2}}}-q^{\frac{1}{2}}(T_2)\right)\\
\notag \leq&\frac{1}{2}C_0-56(C_0+C_1)\frac{1}{(1-\epsilon)^{2}}\frac{1}{m(0)}\\
\notag <&C_0
\end{align}
for any $x\in\mathbb{R}$. Therefore, \eqref{b4} holds throughout the interval $[0,T_2]$. Here the second inequality uses \eqref{b2} and \eqref{c1},
the third inequality uses \eqref{c3}, the fourth inequality uses Lemma \ref{c4},  and the last inequality uses that \eqref{c5} implies that
\begin{align}
\notag -m(0)(1-\epsilon)^2>112\left(1+\frac{C_1}{C_0}\right)
\end{align}
for $\epsilon>0$ sufficiently small. Indeed, $m(0)<-1$ and, $g\geq1$ by hypotheses, and $C_1/C_0<1$ by \eqref{b2}.

For $n=1$, the proof is similar to that in Ref.\cite{r14}, here we include the details for future usefulness.

It follows from \eqref{b7}, where $\delta(t)=q^{\sigma}(t)$, and \eqref{b9}, \eqref{c6} that
\begin{align}\label{c7}
\notag |K_1(t;x)+\phi_1(t;x)|\leq&28(C_1q^{-1}(t)q^{-\frac{\sigma}{2}}(t)+C_2g^2q^{\frac{\sigma}{2}}(t)q^{-1-\sigma}(t))\\
=&28(C_1+C_2g^2)q^{-1-\frac{\sigma}{2}}(t)
\end{align}
for any $t\in[0,T_2]$ for any $x\in\mathbb{R}$. Suppose for now that $v_1(T_2;x)\geq0$. Note from \eqref{a3} that
\begin{align}
\notag \frac{dv_1}{dt}(t;x)=-v_1^2(t;x)-K_1(t;x)-\phi_1(t;x)\leq|K_1(t;x)+\phi_1(t;x)|
\end{align}
for any $(t,x)\in(0,T_2)\times\mathbb{R}$. Integrating this over the interval $[0,T_2]$, we get
\begin{align}
\notag v_1(T_2;x)\leq&\|u^\prime_0\|_{L^\infty(\mathbb{R})}+\int_{0}^{T_2}|K_1(t;x)+\phi_1(t;x)|dt\\
\notag \leq& \frac{1}{2}C_1+28(C_1+C_2g^2)\int_{0}^{T_2}q^{-2}(t)dt\\
\notag \leq& \frac{1}{2}C_1-28(C_1+C_2g^2)\frac{1}{(1-\epsilon)^3m(t)}(q^{-1}(T_2)-(1-\epsilon))\\
\notag \leq& \frac{1}{2}C_1q^{-1}(T_2)-28(C_1+C_2g^2)\frac{1}{(1-\epsilon)^3m(t)}q^{-1}(T_2)\\
\notag \leq& C_1q^{-1}(T_2).
\end{align}
The second inequality uses \eqref{b2} and \eqref{c7}, Lemma \ref{c4}, \eqref{c8}, the third inequality uses \eqref{c3}, 
the fourth inequality uses Lemma \ref{c4}, and the last inequality uses that \eqref{c5} implies that
\begin{align}
\notag -m(0)(1-\epsilon)^3>56\left(1+\frac{C_2}{C_1}g^2\right)
\end{align}
for $\epsilon>0$ sufficiently small. Indeed, $m(0)<-1$ by hypotheses and $C_2/C_1<1/2$ by \eqref{b2}.

Suppose on the other hand that $v_1(T_2;x)<0$. We may assume without loss of generality that 
$\|u^\prime_0\|_{L^\infty}(\mathbb{R})=-m(0)$; we take $-u$ otherwise. It then follows from \eqref{a6} 
and \eqref{b2} that
\begin{align}
\notag v_1(T_2;x)\geq m(T_2)=m(0)q^{-1}(T_2)>-C_1q^{-1}(T_2).
\end{align}
Therefore \eqref{b5} holds throughout the interval $[0,T_2]$.

For $n\geq3$, the proof is similar to that in Ref.\cite{r14}, here we include the details for future usefulness.

For $n\geq2$,  It follows from \eqref{b7} where $\delta(t)=(ng)^{-2}q^{\sigma}(t)$, and \eqref{c6} that
\begin{align}\label{d7}
\notag |K_n(t;x)+\phi_n(t;x)|\leq&28(ngq^{-\frac{\sigma}{2}}(t)C_2((n-1)g)^{2(n-1)}q^{-1-(n-1)\sigma}(t)\\
\notag +&(ng)^{-1}q^{\frac{\sigma}{2}}(t)C_2(ng)^{2n}q^{-1-n\sigma}(t))\\
\notag =&28ngC_2((n-1)g)^{2(n-1)}\left(1+(\frac{n}{n-1})^{2(n-1)}\right)q^{-1-\frac{\sigma}{2}-(n-1)\sigma}(t)\\
 <&28(1+\ep^2)ngC_2((n-1)g)^{2(n-1)}q^{-1-\frac{\sigma}{2}-(n-1)\sigma}(t)
\end{align}
for any $(t,x)\in [0,T_2]\times\mathbb{R}$.

For $n\geq2$, furthermore, let
\begin{align}\label{c9}
v_1(T_{3,n};x)=m(T_{3,n})~~{\rm and}~~m(t)\leq v_1(t;x)\leq \frac{1}{(1+\epsilon)^{1/(2+(n-1)\sigma)}}m(t)
\end{align}
for any $t\in[T_{3,n},T_2]$, for some $T_{3,n}\in(0,T_2)$ and for some $x\in\mathbb{R}$. Indeed, since $v_1$ and $m$ 
are uniformly continuous throughout the interval $[0,T_2]$, we may find $T_{3,n}$ close to $T_2$ so that \eqref{c9} 
holds. Of course,  $x$ depends on $n$, but we suppress it for simplicity of notation. We rerun the argument in the proof of 
Lemma \ref{c4} to arrive at that
\begin{align}\label{d3}
(1+\epsilon)m(0)\leq\frac{dr}{dt}\leq(1-\epsilon)m(0)~~{\rm throughout~the~interval}~(T_{3,n},T_2)
\end{align}
for $\epsilon>0$ sufficiently small, and that
\begin{align}\label{d2}
q(t)\leq r(t)\leq (1+\epsilon)^{1/(2+(n-1)\sigma)}q(t)~~{\rm for~any}~~t\in[T_{3,n},T_2].
\end{align}
Then it follows that
\begin{align}\label{d1}
\notag &\int_{T_{3,n}}^{T_2}q^{-2-(n-1)\sigma}(t)dt\\
\notag \leq&(1+\epsilon)\int_{T_{3,n}}^{T_2}r^{-2-(n-1)\sigma}(t)dt\\
\notag \leq&\frac{1+\epsilon}{1-\epsilon}\frac{1}{m(0)}\int_{T_{3,n}}^{T_2}r^{-2-(n-1)\sigma}(t)\frac{dr}{dt}(t)dt\\
\notag =&-\frac{1}{1+(n+1)\sigma}\frac{1+\epsilon}{1-\epsilon}\frac{1}{m(0)}(r^{-1-(n-1)\sigma}(T_2)-r^{-1-(n-1)\sigma}(T_{3,n}))\\
\leq&-\frac{1}{1+(n+1)\sigma}\frac{1+\epsilon}{1-\epsilon}\frac{1}{m(0)}(r^{-1-(n-1)\sigma}(T_2)-r^{-1-(n-1)\sigma}(T_{3,n})).
\end{align}
This offers refinements over \eqref{c3} when $T_{3,n}$ and $T_2$ are close. Observe that the right side of \eqref{d1} decreases in $n$. Here the first inequality uses \eqref{d2}, the second inequality uses \eqref{d3}, and the last inequality uses \eqref{d2} and \eqref{c9}.

For $n\geq3$, let $|v_n(T_2;x_n)|=\max\limits_{x\in\mathbb{R}}|v_n(T_2;x)|$. Assume without loss of generality that $v_n(T_2;x_n)>0$; we 
take $-u$ otherwise. Choosing $T_{3,n}$ close to $T_2$ so that
\begin{align}\label{d5}
v_n(t;x_n)\geq0~~{\rm for~~any}~~t\in[T_{3,n},T_n].
\end{align}
We necessarily choose $T_{3,n}$ closer to $T_2$ so that \eqref{c9} holds for some $x\in\mathbb{R}$. Consequently, \eqref{d1} 
holds. It follows from \eqref{d4} that
\begin{align}
\notag \frac{dv_n}{dt}(t;x_n)=&-(n+1)v_1(t;x_n)v_n(t;x_n)-\sum^n_{j=1}\binom{n}{j}v_j(t;x_n)v_{n+1-j}(t;x_n)\\
\notag &-K_n(t;x_n)-\phi_n(t;x_n)\\
\notag \leq &-(n+1)m(0)C_2((n-1)g)^{2(n-1)}q^{-1}(t)q^{-1-(n-1)\sigma}(t)\\
\notag &+\sum^n_{j=1}\binom{n}{j}C_2^2((j-1)g)^{2(j-1)}((n-j)g)^{2(n-j)}q^{-1-(j-1)\sigma}(t)q^{-1-(n-j)\sigma}(t)\\
\notag &+|K_n(t;x_n)|+|\phi_n(t;x_n)|\\
\notag \leq &-(n+1)m(0)C_2((n-1)g)^{2(n-1)}q^{-2-(n-1)\sigma}(t)\\
\notag &+\frac{9}{4}\ep^2n((n-1)g)^{2(n-1)}C_2^2q^{-2-(n-1)\sigma}(t)\\
\notag  &+28(1+\ep^2)ngC_2((n-1)g)^{2(n-1)}q^{-1-\frac{\sigma}{2}-(n-1)\sigma}(t)\\
\notag \leq & (-m(0)(n+1)+\frac{3}{2}\ep C_2n+28(1+\ep^2)ng)C_2((n-1)g)^{2(n-1)}q^{-2-(n-1)\sigma}(t)
\end{align}
for any $t\in(T_{3,n},T_2)$. The first inequality uses \eqref{a6}, \eqref{d5} and \eqref{c6}, the second inequality uses 
\eqref{d6} and \eqref{d7}, and the last inequality uses Lemma \ref{c4} and \eqref{c8}. Integrating this over the interval 
$[T_{3,n},T_2]$, we have
\begin{align}
\notag v_n(T_2;x_n)<&v_n(T_{3,n},x_n)+(-m(0)(n+1)+\frac{3}{2}\ep C_2n+28(1+\ep^2)ng)\\
\notag &\times C_2((n-1)g)^{2(n-1)}\int_{T_{3,n}}^{T_2}q^{-2-(n-1)\sigma}(t)dt\\
\notag \leq&C_2((n-1)g)^{2(n-1)}q^{-1-(n-1)\sigma}(T_{3,n})\\
\notag &-(-m(0)(n+1)+\frac{3}{2}\ep C_2n+28(1+\ep^2)ng)\\
\notag &\times\frac{1}{1+(n+1)\sigma}\frac{1+\epsilon}{1-\epsilon}\frac{1}{m(0)}C_2((n-1)g)^{2(n-1)}\\
\notag &\times(q^{-1-(n-1)\sigma}(T_2)-q^{-1-(n-1)\sigma}(T_{3,n}))\\
\notag < &C_2((n-1)g)^{2(n-1)}q^{-1-(n-1)\sigma}(T_{3,n})\\
\notag &+\frac{n+1+\epsilon n}{1+(n-1)\sigma}\frac{1+\epsilon}{1-\epsilon}C_2((n-1)g)^{2(n-1)}\\
\notag &\times(q^{-1-(n-1)\sigma}(T_2)-q^{-1-(n-1)\sigma}(T_{3,n}))\\
\notag < &(1-\frac{4+3\epsilon}{2\sigma+1}\frac{1+\epsilon}{1-\epsilon})C_2((n-1)g)^{2(n-1)}q^{-1-(n-1)\sigma}(T_{3,n})\\
\notag &+\frac{4+3\epsilon}{2\sigma+1}\frac{1+\epsilon}{1-\epsilon}C_2((n-1)g)^{2(n-1)}q^{-1-(n-1)\sigma}(T_2)\\
\notag < &C_2((n-1)g)^{2(n-1)}q^{-1-(n-1)\sigma}(T_2).
\end{align}
Therefore, \eqref{b6} holds for $n=3,4,$ ... throughout the interval $[0,T_2]$. Here the second inequality uses \eqref{c6} 
and \eqref{d1}, the third inequality uses that \eqref{c5} and \eqref{d8} imply that
\begin{align}
\notag -m(0)\epsilon >\frac{3}{2}\ep C_2+28(1+\ep^{1/g})g
\end{align}
for $\epsilon>0$ sufficiently small. Indeed, $m(0)<-1$  by hypotheses and recall \eqref{b2}. The fourth inequality uses 
\eqref{c6} and that $\frac{(1+\epsilon)n+1}{n\sigma+1-\sigma}$ deceases in $n\geq3$, and the last inequality uses \eqref{c6} 
and Lemma \ref{c4}. Indeed,
\begin{align}
0<\frac{4+3\epsilon}{2\sigma+1}\frac{1+\epsilon}{1-\epsilon}<1.
\end{align}

For $n=2$ the proof is similar to that in \cite{r13}, here we include the details for future usefulness.

When $x_2\notin\sum_{1/3}(T_2)$. Let $|v_2(T_2;x_2)|=\max\limits_{x\in\mathbb{R}}|v_2(T_2;x)|$. Assume without loss 
of generality that $v_2(T_2;x_2)>0$. We choose $T_{3,n}$ close to $T_2$ so that
\begin{align}\label{e2}
v_2(t;x_2)\geq0~~{\rm~~for~~any}~~t\in[T_{3,n},T_2].
\end{align}
We necessarily choose $T_{3,n}$ closer to $T_2$ so that \eqref{c9} and, hence, \eqref{d1} hold.

Suppose for now that $x_2\notin\sum_{1/3}(T_2)$, i.e, $v_1(T_2;x_2)>\frac{2}{3}m(T_2)$ (see \eqref{d9}). 
We may necessarily choose $T_{3,n}$ closer to $T_2$ so that
\begin{align}\label{e1}
v_1(t;x_2)\geq\frac{2}{3}m(t)~~{\rm~~for~~any}~~t\in[T_{3,n},T_2].
\end{align}
Indeed, $v_1$ and $m$ are uniformly continuous throughout the interval $[0,T_2]$. The proof is similar to that for 
$n\geq3$. Specifically, it follows from \eqref{d4} that
\begin{align}
\notag \frac{dv_2}{dt}(t;x_2)=&-3v_1(t;x_2)v_2(t;x_2)-K_2(t;x_2)-\phi_2(t;x_2)\\
\notag \leq&-2m(0)C_2g^2q^{-1}(t)q^{-1-\sigma}(t)+28(1+\ep^2)2gC_2g^2q^{-1-\frac{\sigma}{2}-\sigma}(t)\\
\notag \leq&2(-m(0)+28(1+\ep^2)g)C_2g^2q^{-2-\sigma}(t)
\end{align}
for any $t\in(T_{3,n},T_2)$. The first inequality uses \eqref{e1}, \eqref{e2}, \eqref{a8} and \eqref{d7}, and 
the second inequality uses Lemma \ref{c4} and \eqref{c8}. Integrating this over the interval $[T_{3,n},T_2]$, 
we then show that
\begin{align}
\notag v_2(T_2;x_2)<&v_2(T_{3,n};x_2)+2(-m(0)+28(1+\ep^2)g)C_2g^2\int_{T_{3,n}}^{T_2}q^{-2-\sigma}(t)dt\\
\notag \leq&C_2g^2q^{-1-\sigma}(T_{3,n})-2(-m(0)+28(1+\ep^2)g)\\
\notag &\times\frac{1}{1+\sigma}\frac{1+\epsilon}{1-\epsilon}\frac{1}{m(0)}C_2g^2(q^{-1-\sigma}(T_2)-q^{-1-\sigma}(T_{3,n}))\\
\notag \leq&C_2g^2q^{-1-\sigma}(T_{3,n})+\frac{2}{1+\sigma}\frac{(1+\epsilon)^2}{1-\epsilon}C_2g^2(q^{-1-\sigma}(T_2)-q^{-1-\sigma}(T_{3,n}))\\
\notag =&\left(1-\frac{2}{1+\sigma}\frac{(1+\epsilon)^2}{1-\epsilon}\right)C_2g^2q^{-1-\sigma}(T_{3,n})+
\frac{2}{1+\sigma}\frac{(1+\epsilon)^2}{1-\epsilon}C_2g^2q^{-1-\sigma}(T_2)\\
\notag \leq&C_2g^2q^{-1-\sigma}(T_2).
\end{align}
The second inequality uses \eqref{c6} and \eqref{d1}, and the third inequality uses that \eqref{c5} implies that
\begin{align}
\notag -m(0)\epsilon>28(1+\ep^2)g
\end{align}
for $\epsilon>0$ sufficiently small. Indeed, $m(0)<-1$ by hypotheses. The last inequality uses \eqref{c8} 
and Lemma \ref{c4}. Indeed,
\begin{align}
\notag 0<\frac{2}{1+\sigma}\frac{(1+\epsilon)^2}{1-\epsilon}<1.
\end{align}

When $x_2\in\sum_{1/3}(T_2)$. It follows from Lemma \ref{a7} that
\begin{align}\label{f9}
v_1(t;x_2)\leq\frac{2}{3}m(t)<0{\rm~~for~~any}~~t\in[0,T_2].
\end{align}
We shall explore the “smoothing effects” of the solution of \eqref{e3}.

Differentiating \eqref{e3} with respect to $x$ and recalling \eqref{e4}, we have
\begin{align}\label{e7}
\begin{cases}
\frac{d}{dt}(\partial_xX)=v_1(\partial_xX),\\
(\partial_xX)(0 ; x)=1,
\end{cases}
\end{align}
\begin{align}\label{f2}
\begin{cases}
\frac{d}{dt}(\partial^2_xX)=v_2(\partial_xX)^2+v_1(\partial^2_xX),\\
(\partial^2_xX)(0 ; x)=0,
\end{cases}
\end{align}
and
\begin{align}\label{e5}
\begin{cases}
\frac{d}{dt}(\partial^3_xX)=v_3(\partial_xX)^3+3v_2(\partial_xX)(\partial^2_xX)+v_1(\partial^3_xX),\\
(\partial^3_xX)(0;x)=0
\end{cases}
\end{align}
throughout the interval $(0,T_2)$. Integrating \eqref{c2}, moreover, we show that
\begin{align}
\notag v_0(t;x)=u_0(x)-\int_{0}^{t}(K_0(\tau;x)+\phi_0(\tau;x))d\tau
\end{align}
for any $(t, x)\in [0,T_2]\times \mathbb{R}$. Differentiating it with respect to $x$ and recalling \eqref{e4}, then we  arrive at that
\begin{align}\label{f8}
(v_2(\partial_xX)^2+v_1(\partial_x^2X))(t;x)=u^{\prime\prime}_0(x)-I_2(t;x),
\end{align}
\begin{align}\label{e6}
(v_3(\partial_xX)^3+3v_2(\partial_xX)(\partial^2_xX)+v_1(\partial^3_xX))(t;x)=u^{\prime\prime\prime}_0(x)-I_3(t;x)
\end{align}
for any $(t, x)\in [0,T_2]\times \mathbb{R}$, where
\begin{align}\label{f5}
I_2(t;x)=\int_{0}^{t}((K_2+\phi_2)(\partial_xX)^2+(K_1+\phi_1)(\partial^2_xX))(\tau;x)d\tau,
\end{align}
\begin{align}\label{g2}
I_3(t;x)=\int_{0}^{t}((K_3+\phi_3)(\partial_xX)^3+3(K_2+\phi_2)(\partial_xX)(\partial^2_xX)+(K_1+\phi_1)(\partial^3_xX)(\tau;x)d\tau.
\end{align}
Note from \eqref{e5} and \eqref{e6} that
\begin{align}\label{g3}
\begin{cases}
\frac{d}{dt}(\partial^3_xX)(\cdot;x)=u^{\prime\prime\prime}_0(x)-I_3(\cdot ; x),\\
(\partial^3_xX)(0;x)=0.
\end{cases}
\end{align}

We claim that
\begin{align}\label{e9}
\frac{1}{2}q^{1+2\epsilon}(t)\leq(\partial_xX)(t;x_2)\leq2q^{1-\epsilon}(t)~~{\rm for~~any~~}t\in[0,T_2].
\end{align}
Indeed, it follows from \eqref{e3}, \eqref{e7} and \eqref{a9}, \eqref{e8} that
\begin{align}
\notag \frac{1}{1-\epsilon}\frac{dr/dt}{r}\leq\frac{d(\partial_xX)/dt}{\partial_xX}\leq\frac{1}{1+\epsilon}\frac{dr/dt}{r}
\end{align}
throughout the interval $(0,T_2)$. Integrating this over the interval $[0,t]$ and recalling \eqref{e7}, we then show that
\begin{align}
\notag \left(\frac{r(t)}{r(0)}\right)^{1/(1-\epsilon)}\leq(\partial_xX)(t;x_2)\leq\left(\frac{r(t)}{r(0)}\right)^{1/(1+\epsilon)}
\end{align}
for any $t\in[0,T_2]$. Therefore \eqref{e9} follows from \eqref{f1}.

To proceed, we shall show that
\begin{align}\label{f3}
|(\partial^2_xX)(t;x_2)|<-\frac{8}{m(0)}C_2g^2q^{2-\sigma-2\epsilon}(t)
\end{align}
and
\begin{align}\label{f4}
|(\partial^3_xX)(t;x_2)|<\frac{16\epsilon}{m^2(0)}C^2_2g^4q^{3-2\sigma+7\epsilon}(t)
\end{align}
for any $t\in[0,T_2]$. It follows from \eqref{f2} and \eqref{e5} that \eqref{f3} and \eqref{f4} hold at $t=0$. 
Suppose on the contrary that \eqref{f3} and \eqref{f4} hold throughout the interval but do not at $t=T_4$ for some 
$T_4\in(0,T_2]$. By continuity, we find that
\begin{align}\label{f6}
|(\partial^2_xX)(t;x_2)|<-\frac{8}{m(0)}C_2g^2q^{2-\sigma-2\epsilon}(t),
\end{align}
\begin{align}\label{f7}
|(\partial^3_xX)(t;x_2)|<\frac{16\epsilon}{m^2(0)}C^2_2g^4q^{3-2\sigma+7\epsilon}(t)
\end{align}
for any $t\in[0,T_4]$. We seek a contradiction.\\

We use \eqref{f5} to compute that
\begin{align}\label{g1}
\notag |I_2(t;x_2)|\leq&\int_{0}^{t}(112(1+\ep^2)2gC_2g^2q^{-1-\frac{\sigma}{2}-\sigma}(\tau)q^{2-2\epsilon}(\tau)\\
\notag &-\frac{8}{m(0)}C_2g^2q^{2-\sigma-2\epsilon}(\tau)28(C_1+C_2g^2)q^{-1-\frac{\sigma}{2}}(\tau))d\tau\\
\notag \leq&224\left((1+\ep^2)g+2\left(1+\frac{C_2}{C_1}g^2\right)\right)C_2g^2\int_{0}^{t}q^{-\sigma+8\epsilon}(\tau)d\tau\\
\notag \leq&-224\left((1+\ep^2)g+2\left(1+\frac{C_2}{C_1}g^2\right)\right)C_2g^2\\
\notag  &\times\frac{1}{\sigma-1-8\epsilon}\frac{1}{(1-\epsilon)^{\sigma+1-8\epsilon}}\frac{1}{m(0)}(q^{1-\sigma+8\epsilon}(t)-(1-\epsilon)^{\sigma-1+8\epsilon})\\
 <&\epsilon C_2g^2q^{1-\sigma+8\epsilon}(t)
\end{align}
for any $t\in[0,T_4]$. The first inequality uses \eqref{d7}, \eqref{e9} and \eqref{c7}, \eqref{f6} and the second inequality uses Lemma \ref{c4}
and \eqref{c8}. Assume without loss of generality that $\|u^\prime_0\|_{L^\infty(\mathbb{R})}=-m(0)$; we take $-u$ 
otherwise. The third inequality use \eqref{c3}, and the last inequality uses that \eqref{c5} and \eqref{c8} imply
\begin{align}
\notag -m(0)\epsilon(1-\epsilon)^{\sigma+1-8\epsilon}>\frac{224}{\sigma-1-8\epsilon}\left((1+\ep^2)g+2(1+\frac{C_2}{C_1}g^2)\right)
\end{align}
for $\epsilon>0$  sufficiently small. Indeed, $\sigma+1-8\epsilon=5/2-2\epsilon$ and $\sigma-1-8\epsilon=1/2-2\epsilon$ by \eqref{c8}, 
$m(0)<-1$ by hypotheses, $C_2/C_1<1/2$ by \eqref{b2}, and replace $\epsilon$ by $\epsilon/18$. Evaluating \eqref{f8} at $t=T_4$ and $x=x_2$, 
we then show that
\begin{align}
\notag |(\partial^2_xX)(T_4;x_2)|=&|v_1^{-1}(T_4;x_2)||u^{\prime\prime}_0(x_2)-I_2(T_4;x_2)-v_2(T_4;x_2)(\partial_xX)(T_4;x_2)^2|\\
\notag <&-\frac{3}{2}\frac{1}{m(0)}q(T_4)(g^2+\epsilon^2C_2g^2q^{1-\sigma+8\epsilon}(T_4)+4C_2g^2q^{-1-\sigma}(T_4)q^{2-2\epsilon}(T_4))\\
\notag \leq&-\frac{3}{2}(5+\epsilon)\frac{1}{m(0)}C_2g^2q^{2-\sigma-2\epsilon}(T_4)\\
\notag <&-\frac{8}{m(0)}C_2g^2q^{2-\sigma-2\epsilon}(T_4).
\end{align}
Therefore, \eqref{f3} holds throughout the interval $[0,T_2]$. Here the first inequality uses \eqref{f9}, \eqref{a6} and \eqref{b3}, \eqref{g1}, \eqref{c6}, \eqref{e9}, the second inequality uses \eqref{a1}, \eqref{b2} and Lemma \ref{c4}, \eqref{c8}, and the last  inequality follows for $\epsilon>0$  sufficiently small.

Similarly, we use \eqref{g2} to compute that
\begin{align}\label{g4}
\notag |I_3(t;x_2)|<&\int_{0}^{t}(224(1+\ep^2)3gC_2(2g)^4q^{-1-\frac{\sigma}{2}-2\sigma}(\tau)q^{3-3\epsilon}(\tau)\\
\notag &-48\times28(1+\ep^2)2gC_2^2g^4\frac{1}{m(0)}q^{-1-\frac{\sigma}{2}-\sigma}(\tau)q^{1-\epsilon}(\tau)q^{2-\sigma-2\epsilon}(\tau)\\
\notag &+28(C_1+C_2g^2)\frac{\epsilon}{m^2(0)}C_2^2(2g)^4q^{-1-\frac{\sigma}{2}}(\tau)q^{3-2\sigma+7\epsilon}(\tau))d\tau\\
\notag \leq&56\left(12(1+\ep^2)g(\frac{1}{C_2}-\frac{1}{4m(0)})-(1+\frac{C_2}{C_1}g^2)\frac{\epsilon}{m(0)}\right)\\
\notag &\times C_2^2(2g)^4\int_{0}^{t}q^{1-2\sigma+7\epsilon}(\tau)d\tau\\
\notag \leq&-56\left(12(1+\ep^2)g(\frac{1}{C_2}-\frac{1}{4m(0)})-(1+\frac{C_2}{C_1}g^2)\frac{\epsilon}{m(0)}\right)\\
\notag &\times\frac{1}{2\sigma-2-7\epsilon}\frac{1}{(1-\epsilon)^{2\sigma-7\epsilon}}\frac{1}{m(0)}C_2^2(2g)^4(q^{2-2\sigma+7\epsilon}(t)-(1-\epsilon)^{2\sigma-2-7\epsilon})\\
\leq&-\frac{\epsilon^2}{m(0)}C_2^2(2g)^4q^{2-2\sigma+7\epsilon}(t)
\end{align}
for any $t\in[0,T_4]$. The first inequality uses \eqref{d7}, \eqref{e9}, \eqref{f6} and \eqref{c1}, \eqref{f7}, and 
the second inequality uses that \eqref{c8} implies that $2-5/2\sigma-3\epsilon>1-2\sigma+7\epsilon$. Assume without loss of generality that $\|u^\prime_0\|_{L^\infty(\mathbb{R})}=-m(0)$. The third inequality uses \eqref{c3} 
and the last inequality uses that \eqref{c5} implies that
\begin{align}
\notag \epsilon^2(1-\epsilon)^{2\sigma-7\epsilon}(-m(0))^{3/4}>\frac{1680}{2\sigma-2-7\epsilon}(1+\ep^2)
\end{align}
and
\begin{align}
\notag -\epsilon(1-\epsilon)^{2\sigma-7\epsilon}m(0)>\frac{56}{2\sigma-2-7\epsilon}\left(1+\frac{C_2}{C_1}g^2\right)
\end{align}
for $\epsilon>0$ sufficiently small. Indeed, $2\sigma-7\epsilon=3+5\epsilon$ and $2\sigma-2-7\epsilon=1+5\epsilon$ by \eqref{c8}, 
$m(0)<-1$ by hypotheses, $C_2/C_1<1/2$ by \eqref{b2}, and replace $\epsilon$ by $\epsilon/32$. Integrating \eqref{g3} over the interval 
$[0,T_4]$, we then show that
\begin{align}
\notag |(\partial^3_xX)(T_4;x_2)|\leq& \int_{0}^{T_4}(|u^{\prime\prime\prime}_0(x_2)|+|I_3(t;x_2)|)dt\\
\notag <&\int_{0}^{T_4}\left((2g)^4-\frac{\epsilon^2}{m(0)}C_2^2(2g)^4q^{2-2\sigma+7\epsilon}(t)\right)dt\\
\notag \leq& -\left(\frac{1}{C_2^2}-\frac{\epsilon^2}{m(0)}\right)\frac{1}{2\sigma-3-7\epsilon}\frac{1}{(1-\epsilon)^{2\sigma-1-7\epsilon}}\frac{1}{m(0)}\\
\notag &\times C_2^2(2g)^4(q^{3-2\sigma+7\epsilon}(T_4)-(1-\epsilon)^{2\sigma-3-7\epsilon})\\
\notag <&\frac{\epsilon}{m^2(0)}C_2^2(2g)^4q^{3-2\sigma+7\epsilon}(T_4).
\end{align}
Therefore, \eqref{f4} holds throughout the interval $[0,T_2]$. Here the second inequality uses \eqref{b3} 
and \eqref{g4}, the third inequality uses \eqref{c3}, and the last inequality uses that \eqref{c5} implies that
\begin{align}
\notag \epsilon^2(1-\epsilon)^{2\sigma-1-7\epsilon}(-m(0))^{1/2}>\frac{2}{5}
\end{align}
for $\epsilon>0$ sufficiently small, satisfying $(1-\epsilon)^{2\sigma-1-7\epsilon}>2/5$. Indeed, $2\sigma-1-7\epsilon=2+5\epsilon$ and 
$2\sigma-3-7\epsilon=5\epsilon$ by \eqref{c8}, $m(0)<-1$  by hypotheses, and recall \eqref{b2}.

To proceed, since $v_2(T_2;x_2)=\max\limits_{x\in\mathbb{R}}|v_2(T_2;x)|$, it follows that
\begin{align}
\notag v_3(T_2;x_2)(\partial_xX)(T_2;x_2)=0.
\end{align}
Multiplying  \eqref{f8} by $3v_2(\partial_xX)$ and \eqref{e6} by $v_1$ and we take their difference to show that
\begin{align}
\notag v_2^2(T_2;x_2)=&\frac{1}{3}(\partial_xX)^{-3}(T_2;x_2)(v_1^2(T_2;x_2)(\partial^3_xX)(T_2;x_2)\\
\notag &+3v_2(T_2;x_2)(\partial_xX)(T_2;x_2)(u^{\prime\prime}_0(x_2)-I_2(T_2;x_2))\\
\notag &-v_1(T_2;x_2)(u^{\prime\prime\prime}_0(x_2)-I_3(T_2;x_2)))\\
\notag <&\frac{8}{3}q^{-3-6\epsilon}(T_2)(m^2(0)\frac{\epsilon}{m^2(0)}C^2_2(2g)^4q^{-2}(T_2)q^{3-2\sigma+7\epsilon}(T_2)\\
\notag &+6C_2g^2q^{-1-\sigma}(T_2)q^{1-\epsilon}(T_2)(g^2+\epsilon C_2g^2q^{1-\sigma+8\epsilon}(T_2))\\
\notag &-m(0)q^{-1}(T_2)((2g)^4-\frac{\epsilon^2}{m(0)}C_2^2(2g)^4q^{2-2\sigma+7\epsilon}(T_2)))\\
\notag <&\frac{8}{3}\left(\epsilon+\frac{3}{8}(\frac{1}{C_2}+\epsilon)+\left(\frac{1}{(-m(0))^{1/2}}+\epsilon^2\right)\right)C_2^2(2g)^4q^{-2-2\sigma+\epsilon}(T_2)\\
\notag <&C_2^2g^{4}q^{-2-2\sigma}(T_2).
\end{align}
Therefore, \eqref{b6} holds for $n=2$ throughout the interval $[0,T_2]$. Here the first inequality uses \eqref{e9}, 
\eqref{f9}, \eqref{a6}, \eqref{f7}, \eqref{c6} and \eqref{b3}, \eqref{g1}, \eqref{g4}, the second inequality uses \eqref{a1}, \eqref{b2} 
and Lemma \ref{c4}, \eqref{c8}, and the last inequality uses that
\begin{align}
\notag \epsilon +\frac{3}{8}(\epsilon^{\frac{3}{4}}+\epsilon)+\epsilon^{\frac{1}{2}}+\epsilon^2<\frac{3}{2^7}
\end{align}
for $\epsilon>0$ sufficiently small.

To summarize, a contradiction proves that \eqref{b4}, \eqref{b5} and \eqref{b6} hold for any $n=0,1,2,$... throughout the interval $[0,T_1]$.

To proceed, note that
\begin{align}
\notag |K_1(t;x)+\phi_1(t;x)|\leq&28(C_1+C_2g^2)q^{-1-\frac{\sigma}{2}}(t)\\
\notag <&28(C_1+C_2g^2)q^{2}(t)\\
\notag <&28(C_1+C_2g^2)m^{-2}(0)m^2(t)\\
\notag <&\epsilon^2m^2(t)
\end{align}
for any $t\in[0,T_1]$ for any $x\in\mathbb{R}$. The first inequality uses \eqref{c7}, the second inequality uses Lemma \ref{c4} 
and \eqref{c8}, and the third inequality uses \eqref{a6}. Assume without loss of generality that 
$\|u^\prime_0\|_{L^\infty(\mathbb{R})}=-m(0)$. The last inequality uses that \eqref{c5} implies that
\begin{align}
\notag -m(0)\epsilon^2>56\left(1+\frac{C_2}{C_1}g^2\right)
\end{align}
for $\epsilon>0$ sufficiently small. Indeed, $m(0)<-1$ by hypotheses, and $C_2/C_1<1/2$ by \eqref{b2}. 
A contradiction therefore proves \eqref{g5}. Furthermore, \eqref{b4}, \eqref{b5} and \eqref{b6} 
hold for any $n=0,1,2,$... throughout the interval $[0,T^\prime]$ for any $T^\prime<T$.

To conclude, let $x\in\sum_\epsilon(t)$ for $t\in[0,T)$. It follows from \eqref{a9} and \eqref{e8} that
\begin{align}
\notag m(0)(v_1^{-1}(0;x)+(1+\epsilon)t)\leq r(t;x)\leq m(0)(v_1^{-1}(0;x)+(1-\epsilon)t).
\end{align}
Moreover, it follows from Lemma \ref{a7} that $m(0)<v_1(0;x)\leq(1-\epsilon)m(0)$. Hence, one has
\begin{align}
\notag 1+m(0)(1+\epsilon)t\leq r(t;x)\leq \frac{1}{1-\epsilon}+m(0)(1-\epsilon)t.
\end{align}
Furthermore, it follows from \eqref{f1} that
\begin{align}
\notag (1-\epsilon)+m(0)(1-\epsilon^2)t\leq q(t)\leq \frac{1}{1-\epsilon}+m(0)(1-\epsilon)t.
\end{align}
Since the function on the left side decreases to zero as $t\rightarrow -\frac{1}{m(0)}\frac{1}{1+\epsilon}$ 
and since the function on the right side decreases to zero as $t\rightarrow -\frac{1}{m(0)}\frac{1}{(1-\epsilon)^2}$, therefore, 
$q(t)\rightarrow 0$ and, hence (see \eqref{a6}), 
\begin{equation*}
m(t)\rightarrow -\infty ~~{\rm as}~~ t\rightarrow T^-, 
\end{equation*}
where $T$ satisfies \eqref{g6}. 
On the other hand, \eqref{b4} dictates that $v_0(t;x)$ remains bounded for any $t\in[0,T^\prime]$, $T^\prime<T$, for any $x\in\mathbb{R}$. 
In other words, 
\begin{equation*}
\inf\limits_{x\in\mathbb{R}}u_x(t,x)\rightarrow-\infty~~ {\rm as}~~ t\rightarrow T^-,
\end{equation*}
but $u(t,x)$ is bounded for any  $(t, x)\in[0,T)\times\mathbb{R}$. This completes the proof.

\section*{Acknowledgment}
This work is supported by Yunnan Fundamental Research Projects (Grant NO. 202101AU070029).

\section*{Data availability statement}
Data sharing not applicable to this article as no datasets were generated or analysed during the current study.

\section*{Conflict of interest }
The authors declare that they have no conflict of interest.

\end{document}